\documentclass{article}
\usepackage[utf8]{inputenc}
\usepackage[english]{babel}
\usepackage{tikz-cd}
\usepackage[titletoc]{appendix}
\usepackage[nottoc]{tocbibind}

\usepackage[linewidth=1pt]{mdframed}

\usepackage[T1]{fontenc}
\usepackage{color}

\definecolor{identifiercolor}{rgb}{.4,.6,.56}
\definecolor{stringcolor}{gray}{0.5}
\definecolor{inactivecolor}{rgb}{0.15,0.15,0.5}
\usepackage{listings}
\usepackage[margin=1in]{geometry}

\title{A Categorification of Biquandle Brackets}
\author{Adu Vengal \\ \texttt{vengal.8@osu.edu} \and Vilas Winstein \\ \texttt{winstein.1@osu.edu}}
\date{April 2020}

\usepackage[numbers]{natbib}
\usepackage{graphicx}
\usepackage{amssymb,amsmath,dsfont}
\usepackage{amsthm}
\usepackage{comment}

\usepackage{hyperref}
\hypersetup{
    colorlinks=true,
    linkcolor=blue,
    filecolor=blue,      
    urlcolor=blue,
    citecolor=blue,
}

\DeclareMathOperator{\ovtri}{\overline{\rhd}}
\DeclareMathOperator{\untri}{\underline{\rhd}}

\newcommand{\gdim}{\operatorname{gdim}}
\newcommand{\rank}{\operatorname{rank}}
\newcommand{\im}{\operatorname{im}}
\newcommand{\Bh}{\operatorname{Bh}}
\newcommand{\Kh}{\operatorname{Kh}}

\newtheorem{theorem}{Theorem}
\newtheorem{lemma}{Lemma}
\newtheorem{corollary}{Corollary}
\newtheorem{prop}{Proposition}

\theoremstyle{definition}
\newtheorem{definition}{Definition}
\newtheorem{example}{Example}
\newtheorem{remark}{Remark}

\def\semicolon{;}
\def\applytolist#1{
    \expandafter\def\csname multi#1\endcsname##1{
        \def\multiack{##1}\ifx\multiack\semicolon
            \def\next{\relax}
        \else
            \csname #1\endcsname{##1}
            \def\next{\csname multi#1\endcsname}
        \fi
        \next}
    \csname multi#1\endcsname}

\def\calc#1{\expandafter\def\csname c#1\endcsname{{\mathcal #1}}}
\applytolist{calc}QWERTYUIOPLKJHGFDSAZXCVBNM;
\def\bbc#1{\expandafter\def\csname bb#1\endcsname{{\mathbb #1}}}
\applytolist{bbc}QWERTYUIOPLKJHGFDSAZXCVBNM;
\def\bfc#1{\expandafter\def\csname bf#1\endcsname{{\mathbf #1}}}
\applytolist{bfc}QWERTYUIOPLKJHGFDSAZXCVBNM;
\def\sfc#1{\expandafter\def\csname s#1\endcsname{{\sf #1}}}
\applytolist{sfc}QWERTYUIOPLKJHGFDSAZXCVBNM;
\def\fc#1{\expandafter\def\csname f#1\endcsname{{\mathfrak #1}}}
\applytolist{fc}QWERTYUIOPLKJHGFDSAZXCVBNMag;
 
\usepackage{xcolor}

\title{Categorifying Biquandle Brackets}
\author{Adu Vengal \\ \texttt{vengal.8@osu.edu} \and Vilas Winstein \\ \texttt{winstein.1@osu.edu}}
\date{May 2020}

\begin{document}

\maketitle

\begin{abstract}
In their paper entitled ``Quantum Enhancements and Biquandle Brackets,'' Nelson, Orrison, and Rivera introduced \emph{biquandle brackets}, which are customized skein invariants for biquandle-colored links.
These invariants generalize the Jones polynomial, which is categorified by Khovanov homology.
At the end of their paper, Nelson, Orrison, and Rivera asked if the methods of Khovanov homology could be extended to obtain a categorification of biquandle brackets.

We outline herein a Khovanov homology-style construction that is an attempt to obtain such a categorification of biquandle brackets.
The resulting knot invariant generalizes Khovanov homology, but the biquandle bracket is not always recoverable, meaning the construction is not a true categorification of biquandle brackets.
However, the construction does lead to a definition that gives a ``canonical'' biquandle 2-cocycle associated to a biquandle bracket, which, to the authors' knowledge, was not previously known. Though this 2-cocycle is derived solely from its biquandle bracket, the two corresponding invariants are in general incomparable in strength.

We also provide Mathematica packages that can be used to do computations with biquandles, biquandle brackets, biquandle 2-cocycles, and, in particular, the canonical biquandle 2-cocycle associated to a biquandle bracket.
These Mathematica packages can be found at \href{http://www.vilas.us/biquandles/}{vilas.us/biquandles}.
\end{abstract}

\section{Introduction}
\label{sec_intro}

\emph{Khovanov homology} is a link invariant which categorifies the Jones polynomial.
The values of this invariant are not polynomials, but rather sequences of modules obtained from the cohomology of a certain cochain complex.
Khovanov homology categorifies the Jones polynomial because, when a particular quantity (the graded Euler characteristic) is measured from the sequence of modules in the value of the invariant, one recovers the Jones polynomial.
For more details and a construction of Khovanov homology, see \cite{khovanov2000categorification} or \cite{bar2002khovanov}.

\emph{Biquandles} are a type of algebraic structure whose axioms parallel the Reidemeister moves of knot theory.
Because of this, biquandles are the basis for many invariants of knots and links.
In particular, the \emph{biquandle counting invariant} is simply the number of ways to color a link diagram with elements of a biquandle so that relationships between colors at crossings are satisfied.
In \cite{nelson2017quantum}, Sam Nelson et. al. introduced an enhancement of the biquandle counting invariant, called the \emph{biquandle bracket}.
This is a type of skein relation depending on biquandle colorings.
A \emph{biquandle 2-cocycle} is another type of function on a biquandle that can be used to define a link invariant, arising from cohomology theory.

Biquandle brackets generalize the Jones polynomial in a natural way.
In \cite{nelson2017quantum}, Nelson et. al. asked whether a Khovanov homology-style categorification of the biquandle bracket is possible.
Herein, we provide a construction of what seems (to the authors) to be the most natural step from Khovanov homology toward a categorification of biquandle brackets.
The invariant we obtain generalizes Khovanov homology, but it is not a true categorification of all biquandle brackets: in some cases, the biquandle bracket cannot be recovered from our invariant via the graded Euler characteristic.

Nevertheless, the invariant does lead to a way of assigning a biquandle 2-cocycle to any given biquandle bracket, and the relationship and power of the invariant associated with this new biquandle 2-cocycle may be of interest.
To this end, we provide some Mathematica packages that can be used to do experimentations with biquandles, biquandle brackets, and biquandle 2-cocycles, including this new canonical biquandle 2-cocycle associated with a biquandle bracket.

This paper is structured as follows.\
In section \ref{chap_prelim}, we review definitions which we will require for our results, including the definition of biquandles, biquandle brackets, and biquandle 2-cocycles.
In section \ref{chap_construction}, we present the construction of our new invariant, which we call ``biquandle bracket homology,'' and we provide a proof of its invariance as well as an explanation of the canonical 2-cocycle associated with a biquandle bracket.
At the end of section \ref{chap_construction}, we also provide an example computation of the biquandle bracket homology which helps clarify the definitions and statements of the paper.
Section \ref{sec_questions} provides some questions for further inquiry.

This work has been done as a part of the undergraduate research program
\href{http://www.math.ohio-state.edu/~chmutov/wor-gr-su19/wor-gr.htm}{``\underline{Knots and Graphs}''}
at the Ohio State University, during the summer of 2019.
It is a continuation of work done in the same program during the summer of 2018 in \cite{hoffer2019structure}.
We are grateful to the OSU Honors Program Research Fund and to the  NSF-DMS \#1547357 RTG grant: Algebraic Topology and Its Applications for financial support.
In addition, we are grateful to our advisor, Sergei Chmutov, for his help.

\section{Preliminaries}
\label{chap_prelim}

To establish our notation and introduce the topics, we provide the following definitions.
We follow the notation and conventions in \cite{nelson2017quantum}.
These preliminaries can also be found in the previous work \cite{hoffer2019structure}.

\subsection{Biquandles}

\begin{definition}
\label{def_biquandle}
    A \emph{biquandle} is a set $X$ with two binary operations $\untri, \ovtri$ such that $\forall  x, y, z \in X$,
    
    \begin{enumerate}
        \item[(i)]
            $x \untri x = x \ovtri x$
        
        \item[(ii)]
            The maps $\alpha_y(x) = x \ovtri y,\, \beta_y(x) = x \untri y$, and $S(x,y) = (y \ovtri x, x \untri y)$ are invertible. 
        
        \item[(iii)]
            The following exchange laws are satisfied:
            \begin{align*}
                (x\untri y) \untri (z \untri y) &= (x\untri z) \untri (y \ovtri z) \\
                (x\untri y) \ovtri (z \untri y) &= (x \ovtri z) \untri (y \ovtri z) \\
                (x\ovtri y) \ovtri (z \ovtri y) &= (x\ovtri z) \ovtri (y \untri z).
            \end{align*}
    \end{enumerate}
    
    If $x \ovtri y = x$ for all $x, y \in X$, then $X$ is called a \emph{quandle}.
    When there is no danger of confusion, we will write the biquandle $(X,\ovtri,\untri)$ simply as $X$.
\end{definition}

\begin{remark}
\label{rmk_biquandlepresentationmatrix}
    If $X$ is a finite biquandle, we can represent all of the information about it in two operation tables.
    Fix some ordering on the elements of $X$ and label them with the integers $1$ through $n$ (where $n$ is the size of $X$).
    Then the operation table for $\untri$ is an $n \times n$ matrix of integers in $\{ 1, \dotsc, n \}$, and the $(i,j)$ entry of this matrix is $i \untri j$.
    The operation table for $\ovtri$ is defined similarly.
    For example, the following operation tables represent a biquandle on three elements.
    \begin{align*}
        \untri: \quad
        \begin{bmatrix}
            2 & 1 & 2 \\
            1 & 3 & 3 \\
            3 & 2 & 1
        \end{bmatrix}
        \qquad \qquad
        \ovtri: \quad
        \begin{bmatrix}
            3 & 3 & 3 \\
            2 & 2 & 2 \\
            1 & 1 & 1
        \end{bmatrix}
    \end{align*}
\end{remark}

    The conditions in the biquandle definition are analogous to the Reidemeister moves in knot theory when we interpret $x \untri y$ as ``$x$ passing under $y$'' and $x \ovtri y$ as ``$x$ passing over $y$'' in the following way:
    
    \begin{center}
        \includegraphics[scale=0.15]{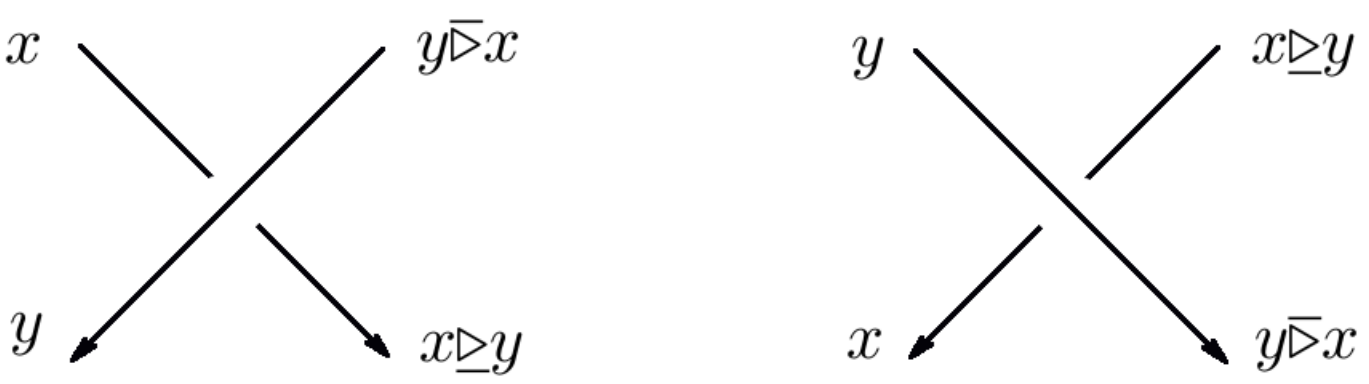}
    \end{center}
    
    Fix a biquandle $X$.
    An $X$-coloring of an oriented knot (or link) diagram $L$ is an assignment of an element of $X$ to each arc in the diagram such that the above relationships hold at each crossing.
    Then the biquandle axioms are precisely what is required for the $X$-coloring to be preserved as Reidemeister moves are performed on the diagram.
    For this reason, the number of $X$-colorings of a diagram is a link invariant, called the \emph{biquandle counting invariant}, and denoted $\Phi_X^\bbZ(L)$.
    
\subsection{Biquandle Brackets}
    
    In \cite{nelson2017quantum}, an enhancement of the biquandle counting invariant is introduced.
    For each $X$-coloring of $D$, one can perform a smoothing operation similar to the construction in the Kauffman bracket, but this time keeping track of the colorings at each crossing as follows:
    
    \begin{center}
        \includegraphics[scale=0.22]{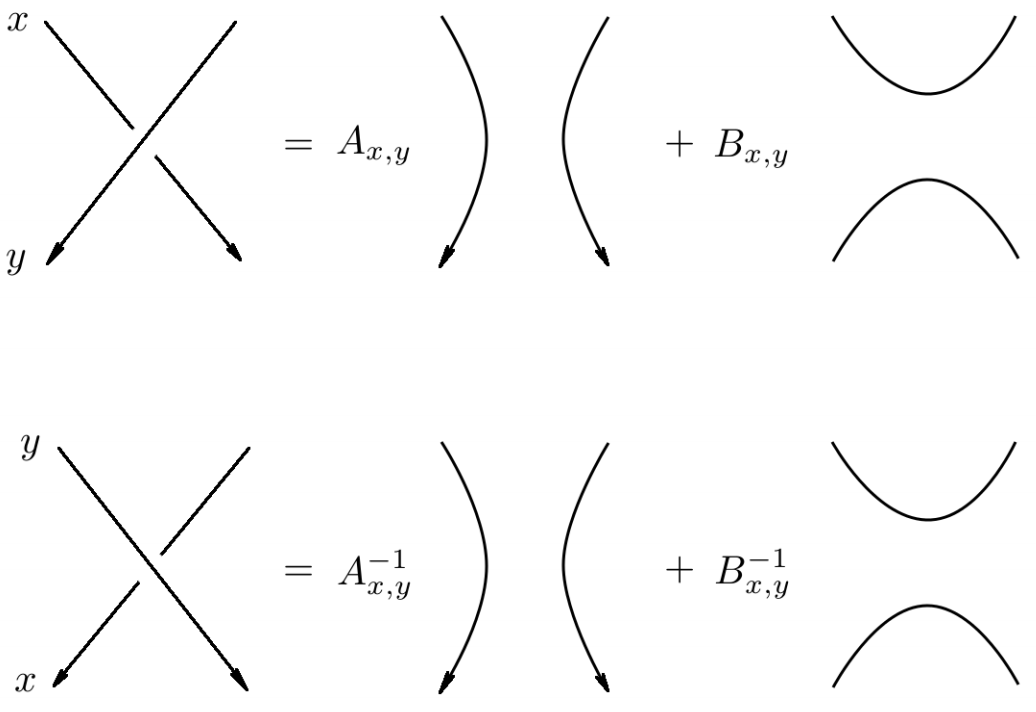}
    \end{center}
    
    Where for each $x,y \in X$, $A_{x,y}$ and $B_{x,y}$ are invertible elements of some commutative ring with unity $R$.
    Additionally, the removal of a circle with no crossings should correspond to multiplication by some element $\delta \in R$, and to correct for the additional states generated by kinks (from the first Reidemeister move), a writhe factor should be included, which can simply be an appropriate power of some element $w \in R^\times$.
    For the bracket to be an invariant of an $X$-colored link, it should not change when Reidemeister moves are applied and the $X$-coloring is updated correspondingly.
    Below are the conditions that must be satisfied by $A,B,\delta$, and $w$ for this to be true.
    For more details, see \cite{nelson2017quantum}.
 
\begin{definition}
\label{def_biquandlebracket}
    A \emph{biquandle bracket} on a biquandle $X$ with values in commutative ring (with unity) $R$ is a pair of maps $A,B:X\times X \rightarrow R^\times$ and two distinguished elements $\delta\in R, w\in R^\times$ which satisfy the following conditions.
    
    \begin{enumerate}
        \item[(i)]
            For all $x \in X$,
            $w=-A_{x,x}^2B_{x,x}^{-1}$.
            
        \item[(ii)]
            For all $x, y \in X$,
            $\delta = -A_{x,y}B_{x,y}^{-1}-A_{x,y}^{-1}B_{x,y}$.
            
        \item[(iii)]
            For all $x, y, z \in X$, all of the following equations hold.
            \begin{align*}
                A_{x,y} A_{y,z} A_{x \untri y, z \ovtri y} &= A_{x,z} A_{y \ovtri x, z \ovtri x} A_{x \untri z, y \untri z}, \\
                A_{x,y} B_{y,z} B_{x \untri y, z \ovtri y} &= B_{x,z} B_{y \ovtri x, z \ovtri x} A_{x \untri z, y \untri z}, \\
                B_{x,y} A_{y,z} B_{x \untri y, z \ovtri y} &= B_{x,z} A_{y \ovtri x, z \ovtri x} B_{x \untri z, y \untri z}, \\
                A_{x,y} A_{y,z} B_{x \untri x, z \ovtri y} &= A_{x,z} B_{y \ovtri x, z \ovtri z} A_{x \untri z, y \untri z} + A_{x, z} A_{y \ovtri x, z \ovtri x} B_{x \untri z, y \untri z} \\
                &\qquad + \delta A_{x, z} B_{y \ovtri x, z \ovtri x} B_{x \untri z, y \untri z} + B_{x, z} B_{y \ovtri x, z \ovtri x} B_{x \untri z, y \untri z}, \\
                B_{x, z} A_{y \ovtri x, z \ovtri x} A_{x \untri z, y \untri z} &= B_{x, y} A_{y, z} A_{x \untri y, z \ovtri y} + A_{x, y} B_{y, z} A_{x \untri y, z \ovtri y} \\
                &\qquad + \delta B_{x, y} B_{y, z} A_{x \untri y, z \ovtri y} + B_{x, y} B_{y, z} B_{x \untri y, z \ovtri y}.
            \end{align*}
    \end{enumerate}
    
    Note that we denote $A(x,y)$ and $B(x,y)$ by $A_{x,y}$ and $B_{x,y}$.
    Additionally, since $\delta$ and $w$ are determined by the maps $A$ and $B$, we will generally denote a biquandle bracket simply by the pair $\beta = (A,B)$.
    Finally, if $\beta$ is a biquandle bracket on a biquandle $X$ taking values in $R$, then we say $\beta$ is an \emph{$X$-bracket}.
\end{definition}

If $f$ is a coloring of an oriented link, the value of the biquandle bracket $\beta$ on $f$ is denoted $\beta(f)$.

\begin{example}
\label{ex_computation}
    Here is a computation of the value of $\beta(f)$ for a coloring $f$ (shown at the top-left corner) of the trefoil knot:
    \begin{center}
        \includegraphics[scale=0.16]{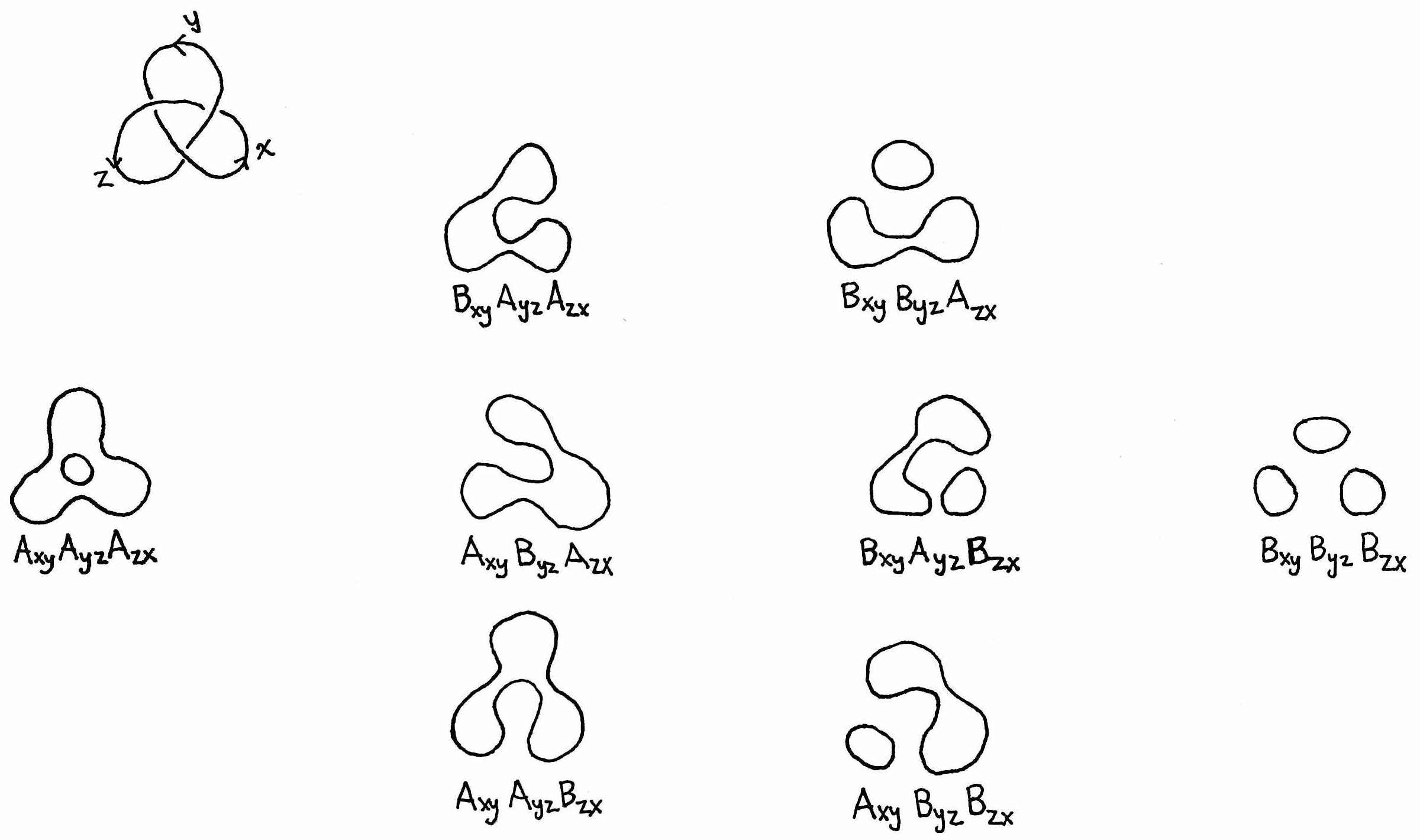}
        {\small
        \[
            \beta(f) = w^{-3} \left( \begin{array}{cccc}
            & + \delta B_{x,y} A_{y,z} A_{z,x} & + \delta^2 B_{x,y} B_{y,z} A_{z,x} & \\
            \delta^2 A_{x,y} A_{y,z} A_{z,x} & + \delta A_{x,y} B_{y,z} A_{z,x} & + \delta^2 B_{x,y} A_{y,z} B_{z,x} & + \delta^3 B_{x,y} B_{y,z} B_{z,x} \\
            & + \delta A_{x,y} A_{y,z} B_{z,x} & + \delta^2 A_{x,y} B_{y,z} B_{z,x} &
            \end{array} \right).
        \]}
    \end{center}
\end{example}

The oriented link invariant corresponding to $\beta$, denoted $\Phi_X^\beta(L)$ simply the multiset of all biquandle bracket values, one for each $X$-coloring of the diagram:
\[
    \Phi_X^\beta(L) = \{ \beta(f) : f \text{ is an } X\text{-coloring of } L \}.
\]
Note that $\Phi_X^\beta(L)$ is an enhancement of the biquandle counting invariant $\Phi_X^\bbZ(L)$ because the counting invariant is simply the cardinality of this multiset.

\begin{remark}
\label{rmk_biquandlebracketpresentationmatrix}
If $X$ is finite and we fix an ordering $X = \left\{ x_1, \dotsc, x_n \right\}$, we can encapsulate all of the information about a biquandle bracket in a presentation matrix.
This is an $n$ by $2n$ matrix $M$ over $R$ with entries $M_{i,j}=A_{x_i,x_j}$ and $M_{i,n+j}=B_{x_i,x_j}$ for $i,j\in\{1,2,...,n\}$.
\end{remark}

\begin{example}
\label{ex_samplebracket}
    Let $X$ be the biquandle given by the following operation table.
    \begin{align*}
        \untri: \quad
        \begin{bmatrix}
            2 & 2 \\
            1 & 1 \\
        \end{bmatrix}
        \qquad \qquad
        \ovtri: \quad
        \begin{bmatrix}
            2 & 2 \\
            1 & 1 \\
        \end{bmatrix}
    \end{align*}
    This biquandle's operations simply flip the left operand, regardless of the right operand.
    Thus, in any $X$-colored link diagram, if one follows a particular strand, the color will alternate at every crossing.
    Let $R = (\mathbb{Z}/2\mathbb{Z})[t] / \left( 1 + t + t^3 \right)$.
    Then the following presentation matrix defines an $X$-bracket.
    \begin{align*}
        \left[
            \begin{array}{cc|cc}
                1 & 1+t & t & t+t^2 \\
                1+t^2 & 1 & 1 & t \\
            \end{array}
        \right]
    \end{align*}
    This biquandle bracket was found in \cite{nelson2017quantum}.
\end{example}

\begin{example}
\label{ex_jones}
    Let $X$ be any biquandle, and let $R$ be any commutative ring.
    If $A_{x,y} = a$ and $B_{x,y} = b$ for all $x,y \in X$ and some $a,b \in R^\times$, then the $X$-bracket $(A,B)$ is called a ``constant'' biquandle bracket (the reader should verify that this does indeed define an $X$-bracket).
    In general, the value of a biquandle bracket on links is unchanged when all values of $A_{x,y}$ and $B_{x,y}$ are scaled by a common factor of $R^\times$ (see \cite{nelson2017quantum}).
    So, dividing through by $b$, the above bracket gives the same invariant as the bracket $A_{x,y} = \frac{a}{b}$, $B_{x,y} = 1$ for all $x,y \in X$. By considering the maps $A, B$ as instead taking values in $R\left[\left(\frac{a}{b}\right)^{\pm 1/2}\right]$, we can make the substitution $q^2 = \frac{a}{b}$ and divide everything through by $q$ to yield the equivalent bracket (when treated over $R\left[\left(\frac{a}{b}\right)^{\pm 1/2}\right]$), $A_{x,y} = q$, $B_{x,y} = q^{-1}$ for all $x, y \in X$. 
    Now, for any particular $X$-coloring of a link, the value of this bracket is evidently the Jones polynomial of the link, which is a Laurent polynomial in the variable $q^2$. Hence the invariant itself still takes values in $R$ rather than $R\left[\left(\frac{a}{b}\right)^{\pm 1/2}\right]$.
    Therefore, the value of a constant biquandle bracket (with $A_{x,y} = a$ and $B_{x,y} = b$) is a multiset containing the Jones polynomial evaluated at $\frac{a}{b}$, and it contains this value with multiplicity equal to the number of $X$-colorings of the link.
    
    Note in the above construction that we only obtain the value of the Jones polynomial \emph{evaluated} at $q \in R$.
    If we want to retain the full power of the Jones polynomial, we can take $R = \bbZ[t,t^{-1}]$ and take $q = t$.
\end{example}

\subsection{Biquandle 2-Cocycles}

Next, we define a biquandle 2-cocycle following the notation of \cite{nelson2017quantum}.
 
\begin{definition}
\label{def_biquandle2cocycle}
    Let $X$ be a biquandle, and let $G$ be an abelian group (written multiplicatively here).
    A function $\phi:X\times X\rightarrow G$ is a \emph{biquandle 2-cocycle} if, for all $x, y, z \in X$, we have
    \begin{enumerate}
        \item[(i)] $\phi(x,x) = 1$,
        \item[(ii)] $\phi(x,y) \cdot \phi(y,z) \cdot \phi\left(x\untri y, z\ovtri y\right) = \phi(x,z) \cdot \phi\left(y\ovtri x,z\ovtri x\right) \cdot \phi\left(x\untri z, y\untri z\right)$.
    \end{enumerate}
\end{definition}

A $2$-cocycle $\phi$ can be used to define the \emph{biquandle 2-cocycle invariant}, as seen in \cite{ceniceros2014augmented}.
Namely, for each $X$-coloring of a link, compute the value
\[
    \prod_{\tau} \phi\left(x_\tau,y_\tau\right)^{\epsilon(\tau)},
\]
where $\tau$ ranges across all crossings in the colored link, $\epsilon(\tau)$ is the sign (either $+1$ or $-1$) of $\tau$, and $x_\tau,y_\tau$ are the biquandle colors of the arcs on the left side of the crossing when it is oriented so that strands point downwards, following a similar convention to the biquandle bracket above.
The value of the biquandle $2$-cocycle invariant associated to $\phi$ is then the multiset of all such values, one for each $X$-coloring of the link.

\begin{remark}
\label{rmk_2cocyclepresentationmatrix}
    Again if $X$ is finite, we construct a presentation matrix for a cocycle in the same fashion as with the biquandle brackets; fixing the ordering $X = \left\{ x_1, \dotsc, x_n \right\}$, the presentation matrix $P$ for a cocycle is an $n \times n$ matrix over $A$ with entries $P_{i,j} = \phi \left( x_i, x_j \right)$.
\end{remark}

\begin{example}
\label{ex_linking}
    Let $X$ be the biquandle described in Example \ref{ex_samplebracket} above.
    Let $A$ be the free abelian group on two symbols, $a$ and $b$.
    Then the following presentation matrix defines a biquandle 2-cocycle $\phi : X \times X \to A$.
    \begin{align*}
        \begin{bmatrix}
            1 & a \\
            b & 1
        \end{bmatrix}
    \end{align*}
    The invariant corresponding to $\phi$ is trivial on all knots (single-component links).
    In fact, more is true: for any $X$-colored knot diagram, at any crossing $\tau$, we have $x_\tau = y_\tau$ (so that $\phi\left(x_\tau, y_\tau\right) = 1$).
    Additionally, for a two-component link, the invariant corresponding to $\phi$ is the multiset $\left\{ 1, 1, \left( a b \right)^\ell, \left( a b \right)^\ell \right\}$, where $\ell$ is the linking number of the two components of the link.
    For a proof of these facts, see example 3 in \cite{hoffer2019structure}.
\end{example}

\section{Biquandle Bracket Homology}
\label{chap_construction}

We now present the construction of a Khovanov homology-style invariant of links based on Khovanov's original construction in \cite{khovanov2000categorification}.

\subsection{Group-Graded Modules and Complexes}

Khovanov's original construction involves taking the homology of chain complexes of $\bbZ$-graded modules.
Calculating the graded Euler characteristic of these homologies yields the Jones polynomial in a variable $q$, where $q$ is the generator of $\bbZ$, and the additive structure of $\bbZ$ is written multiplicatively (i.e. $\bbZ = \{ \dotsc, q^{-2}, q^{-1}, 1, q, q^2, \dotsc \}$).
With the aim of obtaining the biquandle bracket as the graded Euler characteristic of a homology invariant, we need to extend our grading from $\bbZ$ to an arbitrary abelian group.
\begin{definition}
\label{def_Hgradedring}
If $H$ is an abelian group, an \emph{$H$-graded ring} is a ring $S$ which can be decomposed as a direct sum \[ S = \bigoplus_{h \in H} S_h\] of additive groups, such that $S_gS_h\subset S_{gh}$ for any $g,h\in H$.

\end{definition}
\begin{definition}
\label{def_HgradedSmodule}
    If $H$ is a group and $S$ is an $H$-graded commutative ring, an \emph{$H$-graded $S$-module} is an $S$-module $M$ which can be decomposed as a direct sum
    \[
        M = \bigoplus_{h \in H} M_h
    \]
    of additive groups (not necessarily $S$-modules), such that $S_g M_h\subset M_{gh}$ for any $g,h\in H$.
    If $a\in M\setminus\{0\}$, we say $a$ has \emph{well defined degree} if $a\in M_h$ for some $h$. In this case we say $a$ has \emph{degree} $h$, and write $\deg(a)=h$.
\end{definition}

Note that if $S$ is an $H$-graded commutative ring, then $S$ is also an $H$-graded $S$-module, with the $H$-grading given by $S = \bigoplus S_h$.
Also note that direct sums and tensor products of $H$-graded $S$-modules are still $H$-graded $S$-modules.
Indeed, if we have
\[
    M = \bigoplus_{h \in H} M_h \qquad \text{and} \qquad N = \bigoplus_{h \in H} N_h,
\]
then the direct sum can be written as
\[
    M \oplus N = \bigoplus_{h \in H} M_h \oplus N_h,
\]
so an $H$-grading can be given by $(M \oplus N)_h = M_h \oplus N_h$.
Also, the tensor product can be written as
\[
    M \otimes N = \bigoplus_{h \in H} \left( \bigoplus_{g f = h} (M_g \otimes_\bbZ N_f)/\langle sm\otimes n - m\otimes sn: s\in S, m\in M_g, n\in N_f\rangle \right),
\]
and an $H$-grading can be given by
\[
    (M \otimes N)_h = \bigoplus_{g f = h} (M_g \otimes_\bbZ N_f)/\langle sm\otimes n - m\otimes sn: s\in S, m\in M_g, n\in N_f\rangle.
\]
To see why the tensor product takes precisely this form, see \cite{hazrat2016graded}.
\begin{remark}
\label{rmk_quotientgrading}
If $S$ is an $H$-graded ring, and $\varphi:H\to G$ is a group homomorphism, then $S$ is a $G$-graded ring with 
\[
S= \bigoplus_{g\in G} \left(\bigoplus_{h\in\varphi^{-1}(\{g\})} S_h\right)
\]
Similarly if $M$ is an $H$-graded $S$-module, then $M$ is a $G$-graded $S$-module with \[
M= \bigoplus_{g\in G} \left(\bigoplus_{h\in\varphi^{-1}(\{g\})} M_h\right)
\]
Thus if $L\unlhd H$, $M$ is naturally an $H/L$-graded $S$-module via the projection map.
\end{remark}
\begin{definition}
\label{def_gdim}
    If $M = \bigoplus_{h \in H} M_h$ is an $H$-graded $S$-module, the \emph{graded dimension} of $M$ is
    \[
        \gdim(M) = \sum_{h \in H} \rank_\bbZ(M_h) \cdot h.
    \]
    This is a (possibly infinite) formal sum of group elements with coefficients in $\bbZ_{\geq0}\cup\{\infty\}$.
\end{definition}

Note that direct sums of $H$-graded $S$-modules satisfy the following graded dimension equation:
\begin{align*}
    \gdim(M \oplus N) &= \gdim(M) + \gdim(N)
\end{align*}
Also, if $M$ and $N$ are free $S$-modules with finite $S$-rank, having $S$-bases $(m_i)_{i=1}^k$ and $(n_j)_{j=1}^\ell$ respectively, and if $\deg(m_i)=h_i$ and $\deg(n_j)=g_j$ for all $i$ and $j$, then \[\gdim(M\otimes N) = \gdim(S)\sum_{i=1}^k\sum_{j=1}^\ell h_ig_j.\]
\begin{definition}
\label{def_Hgradingshift}
    If $M$ is an $H$-graded $S$-module, we can \emph{shift the grading} of $M$ by an element $h \in H$ and obtain a new $H$-graded $S$-module $M\{h\}$.
    The underlying module structure is the same, but if $a \in M$ has degree $g$, then the same element $a \in M\{h\}$ has degree $hg$. If $\sum_{h\in I} h$ is a formal sum of group elements, for some subset $I\subset H$, we define $M\{\sum_{h\in I} h\} = \bigoplus_{h\in I} M\{h\}$.
\end{definition}

\begin{definition}
\label{def_complex}
    A \emph{cochain complex} of $H$-graded $S$-modules is a sequence $C = \left( C^i \right)_{i \in \bbZ}$ of $H$-graded $S$-modules along with differentials $d^i : C^i \to C^{i+1}$ such that $d^{i+1} \circ d^i = 0$ for all $i \in \bbZ$.
    We say that the differentials are \emph{degree-preserving} if $\deg(d^i(a)) = h$ whenever $\deg(a) = h$.
\end{definition}

\begin{definition}
\label{def_cohomology}
    Suppose $(C,d)$ is a cochain complex of $H$-graded $S$-modules with degree-preserving differentials.
    Then the \emph{cohomology sequence} of $(C,d)$ is the sequence $\cH(C) = (\cH^i)_{i \in \bbZ}$, where
    \[
        \cH^i = \ker(d^{i}) / \im(d^{i-1}).
    \]
    Note that each $\cH^i$ is again an $H$-graded $S$-module, with grading given by
    \[
        (\cH^i)_h = (\ker(d^{i}))_h/(\im(d^{i-1}))_h.
    \]
    This makes sense when $d$ is degree-preserving since $(\ker(d^i))_h \supset (\im(d^{i-1}))_h$, so we have the following decomposition:
    \[
        \cH^i = \ker(d^{i}) / \im(d^{i-1}) = \bigoplus_{h \in H} (\ker(d^{i}))_h/(\im(d^{i-1}))_h.
    \]
\end{definition}

\begin{definition}
\label{def_eulerchar}
    If $C = (C^i)_{i \in \bbZ}$ is a sequence of $H$-graded $S$-modules, the \emph{graded Euler characteristic} of $C$ is
    \[
        \chi(C) = \sum_{i \in \bbZ} (-1)^i \gdim(C^i).
    \]
\end{definition}

Note that if the differentials of a cochain complex $C$ are degree-preserving, then the Euler characteristic of the cohomology sequence is the same as the Euler characteristic of the original complex: $\chi(C) = \chi(\cH(C))$.
This is shown in \cite{hatcher2001algebraic} for ungraded complexes, and the proof can be extended to this setting since, when the differentials are degree preserving, the cochain complex and the cohomology sequence both decompose as graded direct sums of ungraded sequences of modules.

\begin{definition}
\label{def_complexshift}
    If $C = (C^i)_{i \in \bbZ}$ is a cochain complex of $H$-graded $S$-modules, we can \emph{shift the index} of $C$ by an integer $j$ and obtain a new cochain complex $C[j]$.
    Again, the underlying complex structure is the same aside from this shift, we simply have $C[j]^i = C^{i-j}$, and the differential maps are shifted accordingly.
    
    We will also use the notation $C\{h\}$ to denote a cochain complex derived from $C$ by shifting the grading of each $H$-graded $S$-module in $C$ by the element $h \in H$.
    That is, if $C = (C^i)_{i \in \bbZ}$, then $C\{h\} = (C^i\{h\})_{i \in \bbZ}$.
    So $\{ h \}$ always corresponds to a shift of the $H$-grading by the group element $h$, and $[j]$ always corresponds to a shift of the index of the complex by the integer $j$.
\end{definition}

Note that for any complex $C$ and any $h \in H$ and $j \in \bbZ$, we have
\[
    \chi(C[j]\{h\}) = (-1)^j h \chi(C).
\]

\begin{remark}
\label{rmk_evaluation}
Suppose $H$ is a subset of $R^\times$, the group of units in some commutative unital ring $R$, and suppose we have a \emph{finite} formal sum of elements of $H$, i.e. a sum of the form
\[
    \sum_{h \in H} n_h \cdot h,
\]
where each $n_h$ is an integer and only finitely many of them are nonzero.
Then we can \emph{evaluate} this formal sum to obtain the following element of $R$:
\[
    \sum_{h \in H} \iota(n_h) h,
\]
where $\iota : \bbZ \to R$ is the ring homomorphism which sends $1 \in \bbZ$ to $1 \in R$.
\end{remark}

\subsection{The Ring \texorpdfstring{$S$}{S} and the Algebra \texorpdfstring{$M$}{M}}
\label{sec_rsam}

Throughout the rest of the paper, let $X$ be a fixed biquandle with a distinguished element $x_0$, and let $R$ be a fixed commutative ring.
Also let $\beta = (A,B)$ be a fixed $X$-bracket taking values in $R$, which will be the basis for our construction.

For each $x,y \in X$, let
\[
    q_{x,y} = - A_{x,y}^{-1}B_{x,y},
\]
and let $q = q_{x_0,x_0}$.
Let $G$ be the group generated by the elements $q_{x,y}^{-1} q$ of $R^\times$:
\[
    G = \left\langle q_{x,y}^{-1} q : x, y \in X \right\rangle \leq R^\times.
\]
Finally, let $S$ be the $R^\times$ graded group algebra $\bbZ[G]$, with the $R^\times$-grading given by $\deg(g) = g$ for all $g \in G$.

Now let $M$ be the $R^\times$-graded $S$-module $S[t]/(t^2)$ with the additional grading given by $\deg(1) = q$ and $\deg(t) = q^{-1}$.
This means that, for example, the element $g t \in M$ has degree $gq^{-1}$, while the element $g \in M$ has degree $g q$.
$M$ is a Frobenius algebra with the following multiplication and comultiplication operations:
\begin{align*}
    m &: M \otimes M \to M \\
    m &: 1 \otimes 1 \mapsto 1, \qquad 1 \otimes t \mapsto t, \\
    &\quad t \otimes 1 \mapsto t, \qquad t \otimes t \mapsto 0. \\
    \\
    \Delta &: M \to M \otimes M \\
    \Delta &: 1 \mapsto 1 \otimes t + t \otimes 1, \qquad t \mapsto t \otimes t.
\end{align*}
Both of these operations are ``degree-lowering,'' in the sense that the degree of the image of an element is $q^{-1}$ times the degree of the element in the domain.

The ring $S$ and the algebra $M$ are analogous (and will play the same roles) as the $\bbZ$-graded ring $R = \bbZ[c]$ and the $\bbZ$-graded $R$-algebra $A = R[X]/(X^2)$ of \cite{khovanov2000categorification}.
As we will see, however, in the case of a general biquandle, the ring $S$ is more complicated than the ring $R$ of \cite{khovanov2000categorification}.
This complication is what leads to the failure of the cohomology invariant presented here to truly categorify the biquandle bracket.

\subsection{Construction of the Cohomology Invariant}

We are now ready to present the construction of the link invariant.
Suppose $f$ is an $X$-coloring of an oriented link $L$.
Perform smoothings as in the construction of the biquandle bracket link invariant (recall the image from example \ref{ex_computation}, which has been replicated below):

\begin{center}
    \includegraphics[scale=0.16]{SmoothingsPen.jpg}
\end{center}

There will be $2^n$ smoothings in total, where $n$ is the number of crossings in the link diagram.
We arrange these smoothings as the vertices of an $n$-dimensional cube, ordered from left to right by the number of horizontal splittings on positive crossings (those contributing a $B_{x,y}$ factor) and vertical splittings on negative crossings (those contributing an $A_{x,y}^{-1}$ factor). Note that we're still using the convention of orienting crossings so the strands point downwards when determining vertical and horizontal.

Now replace each circle in each smoothing with a copy of $M$ and tensor adjacent copies together.
Shift the $R^\times$-grading of each resultant $S$-module by the coefficient extracted by the smoothing, as follows:

{\small
\begin{center}
    \begin{tikzcd}
                                                                          &  & {\substack{M \\ \left\{-B_{x,y} A_{y,z} A_{z,x} \right\}}}  &  & {\substack{M \otimes M \\ \left\{ B_{x,y} B_{y,z} A_{z,x} \right\}}} &  &                                                                        \\
                                                                          &  & \phantom{\oplus}                                                           &  &                                                                        &  &                                                                        \\
    {\substack{M \otimes M \\ \left\{A_{x,y} A_{y,z} A_{z,x} \right\}}} &  & {\substack{M \\ \left\{ -A_{x,y} B_{y,z} A_{z,x} \right\}}} &  & {\substack{M \otimes M \\ \left\{ B_{x,y} A_{y,z} B_{z,x} \right\}}} &  & {\substack{M \otimes M \otimes M \\ \left\{ -B_{x,y} B_{y,z} B_{z,x} \right\}}} \\
                                                                          &  &                                                            &  & \phantom{\oplus}                                                                       &  &                                                                        \\
                                                                          &  & {\substack{M \\ \left\{ -A_{x,y} A_{y,z} B_{z,x} \right\}}} &  & {\substack{M \otimes M \\ \left\{ A_{x,y} B_{y,z} B_{z,x} \right\}}} &  & \\
                                                                                                                                &  &                                                                                                                 &  &                                                                                                              &  &                                                                                             
    \end{tikzcd}
\end{center}}

Next, we add maps in between the modules to form a cube.
The maps are derived from the Frobenius algebra structure on $M$, which mimics the data of a topological quantum field theory.
Thus, if we'd like to form a map $M \otimes M \to M$, we'll use the multiplication map $m$, and if we'd like to form a map $M \to M \otimes M$, we'll use the comultiplication map $\Delta$.

If two smoothings differ by how a single crossing was split, then the multiplication and comultiplication maps should be applied to the tensor factors of $M$ involved in a cobordism between the two corresponding smoothings.
For example, the upper rightmost map in the following diagram is applied to the tensor factors corresponding with the cobordism taking the circle at the bottom of the top-right smoothing to the pair of circles at the bottom of the right-most smoothing of the link diagram.

Each map must be augmented by multiplication by a particular element of $G$ so that the end result will be a cochain complex with degree-preserving differentials.
The element of $G$ to be multiplied is simply $q$ multiplied by the quotient of the shift applied to the domain over the shift applied to the codomain. Thus if a map corresponds to switching how a crossing colored by $x$ and $y$ was split, then the map must be multiplied by $qq_{x,y}^{-1}$.

Due to the commutation relations involving $m$ and $\Delta$ (see \cite{khovanov2000categorification}), and the fact that these are $S$-module homomorphisms, this yields a cube with commutative faces.
But we will be summing the modules along the columns, and we'd like a cochain complex, so we need the faces of the cube to be \emph{anti}-commutative.
This is why some of the maps in the following image have minus signs.
For a detailed explanation of the position of the minus signs, see Section 3.2 of \cite{bar2002khovanov}.

{\small
\begin{center}
        \begin{tikzcd}
                                                                                                                                                                                  &  & {\substack{M \\ \left\{-B_{x,y} A_{y,z} A_{z,x} \right\}}} \arrow[rr, "{-qq_{y,z}^{-1}\Delta}"'] \arrow[rr] \arrow[rrdd, near start, "{-qq_{z,x}^{-1}\Delta}"']                             &  & {\substack{M \otimes M \\ \left\{ B_{x,y} B_{y,z} A_{z,x} \right\}}} \arrow[rrdd, "{qq_{z,x}^{-1}\Delta}"'] &  &                                                                        \\
                                                                                                                                                                                  &  &  \phantom{\oplus}                                                                                                                                                                      &  & \;                                                                                                           &  &                                                                        \\
        {\substack{M \otimes M \\ \left\{A_{x,y} A_{y,z} A_{z,x} \right\}}} \arrow[rruu, "{qq_{x,y}^{-1}m}"'] \arrow[rrdd, "{qq_{z,x}^{-1}m}"'] \arrow[rr, "{qq_{y,z}^{-1}m}"'] &  & {\substack{M \\ \left\{ -A_{x,y} B_{y,z} A_{z,x} \right\}}} \arrow[rruu, crossing over, very near end, "{qq_{x,y}^{-1}\Delta}"'] &  & {\substack{M \otimes M \\ \left\{ B_{x,y} A_{y,z} B_{z,x} \right\}}} \arrow[rr, "{-qq_{y,z}^{-1}\Delta}"']  &  & {\substack{M \otimes M \otimes M \\ \left\{ -B_{x,y} B_{y,z} B_{z,x} \right\}}} \\
                                                                                                                                                                                  &  &  \phantom{\oplus}                                                                                                                                                                           &  &  \;                                                                                                            &  &                                                                        \\
                                                                                                                                                                                  &  & {\substack{M \\ \left\{ -A_{x,y} A_{y,z} B_{z,x} \right\}}} \arrow[rr, "{qq_{y,z}^{-1}\Delta}"'] \arrow[rruu, very near end, "{qq_{x,y}^{-1}\Delta}"']                                         &  & {\substack{M \otimes M \\ \left\{ A_{x,y} B_{y,z} B_{z,x} \right\}}} \arrow[rruu, "{qq_{x,y}^{-1}\Delta}"'] \arrow[from=lluu, crossing over, near start, "{-qq_{z,x}^{-1}\Delta}"'] &  &   \\
                                                                                                                                &  &                                                                                                                 &  &                                                                                                              &  &                                                                                             
        \end{tikzcd}
\end{center}}

Now take the direct sum of the modules and maps along the columns to obtain a sequence $\tilde{C}_\beta(f)$ of $R^\times$-graded $S$-modules with maps between them.
This in fact constitutes a cochain complex because the faces of the cube are anti-commutative.
Each element in each module will have two images in each target space under the double-differential map, corresponding to the two different ways to get around the square between the domain and target.
And these images will cancel out, leading to $d^{i+1} \circ d^i = 0$.

{\small
\begin{center}
        \begin{tikzcd}
                                                                                                                                        &  & {\substack{M \\ \left\{-B_{x,y} A_{y,z} A_{z,x} \right\}}} \arrow[rr] \arrow[rrdd]                               &  & {\substack{M \otimes M \\ \left\{ B_{x,y} B_{y,z} A_{z,x} \right\}}} \arrow[rrdd]                          &  &                                                                                             \\
                                                                                                                                        &  & \oplus \arrow[d, no head, dotted]                                                                               &  & \oplus \arrow[d, no head, dotted]                                                                            &  &                                                                                             \\
        {\substack{M \otimes M \\ \left\{A_{x,y} A_{y,z} A_{z,x} \right\}}} \arrow[rruu] \arrow[rrdd] \arrow[rr] \arrow[dddd, dotted] &  & {\substack{M \\ \left\{ -A_{x,y} B_{y,z} A_{z,x} \right\}}} \arrow[rruu, crossing over]  \arrow[d, no head, dotted] &  & {\substack{M \otimes M \\ \left\{ B_{x,y} A_{y,z} B_{z,x} \right\}}} \arrow[rr] \arrow[d, no head, dotted] &  & {\substack{M \otimes M \otimes M \\ \left\{ -B_{x,y} B_{y,z} B_{z,x} \right\}}} \arrow[dddd, dotted] \\
                                                                                                                                        &  & \oplus \arrow[d, no head, dotted]                                                                               &  & \oplus \arrow[d, no head, dotted]                                                                            &  &                                                                                             \\
                                                                                                                                        &  & {\substack{M \\ \left\{ -A_{x,y} A_{y,z} B_{z,x} \right\}}} \arrow[rr] \arrow[rruu] \arrow[dd, dotted]           &  & {\substack{M \otimes M \\ \left\{ A_{x,y} B_{y,z} B_{z,x} \right\}}} \arrow[rruu] \arrow[dd, dotted] \arrow[from=lluu, crossing over]      &  &                                                                                             \\
                                                                                                                                        &  &                                                                                                                 &  &                                                                                                              &  &                                                                                             \\
        \tilde{C}^0_\beta(f) \arrow[rr, "d^0"]                                                                                                   &  & \tilde{C}^1_\beta(f) \arrow[rr, "d^1"]                                                                                   &  & \tilde{C}^2_\beta(f) \arrow[rr, "d^2"]                                                                                &  & \tilde{C}^3_\beta(f)                                                                                
        \end{tikzcd}
\end{center}}

Now we must apply a shift to account for the writhe of the original link.
Let $n_+$ be the number of positive crossings in $L$ and let $n_-$ be the number of negative crossings.
Recall that $w = -A_{x_0,x_0}^2B_{x_0,x_0}^{-1} \in R^\times$.
So define $C_\beta(f)$ to be the following shifted cochain complex:
\[
    C_\beta(f) = \tilde{C}_\beta(f)[-n_-] \{ (-1)^{n_-} w^{n_--n_+} \}.
\]
Finally, take the cohomology of $C_\beta(f)$ to obtain a sequence $\cH_\beta(f) = \cH(C_\beta(f))$ of $R^\times$-graded $S$-modules.
This is an invariant of $X$-colored oriented links $f$.
If we want to start with an uncolored oriented link we must consider all $X$-colorings simultaneously as follows:
\[
    \Bh_\beta(L) = \{ \cH_\beta(f) : f \text{ is an } X \text{-coloring of } L \}.
\]
The quantity $\Bh_\beta(L)$ is an invariant of oriented links $L$.
This will be proven in the next section.

\begin{remark}
\label{rmk_smallergradingroup}
For the rest of this remark let $H=\langle A_{x,y}, -B_{x,y} : x,y\in X\rangle\leq R^\times$. Though we've been considering each $C_\beta^k(f)$ to be $R^\times$-graded, each $C_\beta^k(f)$ actually has no nonzero elements of degree $r$ for $r\notin H$. To see this, first note that $q_{x,y}=-B_{x,y}A_{x,y}^{-1}\in H$, so $G= \langle q_{x,y}^{-1}q:x,y\in X\rangle \leq H$. Thus $S=\bbZ[G]$ has no nonzero elements of degree $r$ for $r\notin H$, and so $S=\bbZ[G]$ is actually an $H$-graded ring. $M=S[t]/(t^2)$ with $\deg(1),\deg(t)\in H$, so then $M$ is an $H$-graded $S$-module. Thus any tensor power of $M$ is an $H$-graded $S$-module. Next, $C_\beta^{-n_-}(f)$ is a tensor power of $M$, with degree shifted by
\[
(-1)^{n_-}w^{n_--n_+}\prod_{\tau_+} A_{x_\tau,y_\tau}\prod_{\tau_-} B_{x_\tau,y_\tau}^{-1} = A_{x_0,x_0}^{2n_--2n_+}(-B_{x_0,x_0})^{n_+-n_-}\prod_{\tau_+} A_{x_\tau,y_\tau}\prod_{\tau_-} (-B_{x_\tau,y_\tau})^{-1}\in H
\]
where $\tau_+$ (respectively $\tau_-$) ranges over all positive (respectively negative) crossings of the diagram, and $x_\tau,y_\tau$ are the colors of the strands at the crossing. Thus $C_\beta^{n_-}(f)$ is an $H$-graded $S$-module. By way of induction, suppose that each of the direct summands of $C_\beta^{k}$ are $H$-graded $S$-modules. Then each of the direct summands of $C_\beta^{k+1}(f)$ is a tensor power of $M$ with degree shifted by $-B_{x,y}A_{x,y}^{-1}$ multiplied by the degree shift of one of the direct summands of $C_\beta^k(f)$. Since this shift is in $H$, the direct summands of $C_\beta^{k+1}(f)$ are $H$-graded $S$-modules. Hence $(C^k_\beta(f))$ is a cochain complex of $H$-graded $S$-modules, and so $\cH_\beta(f)$ is a sequence of $H$-graded $S$-modules.
\end{remark}

\begin{example}
\label{ex_khovanovhomology}
    When we take $\beta$ to be any constant biquandle bracket as in example \ref{ex_jones}, we have $q_{x,y} = q$ for all $x,y \in X$.
    Suppose we also started with $R = \bbZ[t,t^{-1}]$ and $q = t$ so that the biquandle bracket value is the same as the Jones polynomial. Since the value of $\beta$ on any link doesn't change upon multiplying each $A_{x,y}$ and $B_{x,y}$ by the same constant \cite{nelson2017quantum}, we may assume $A_{x,y}=1$ and $B_{x,y}=-t$ for all $x,y\in X$. Then, using remark \ref{rmk_smallergradingroup}, $\langle A_{x,y}, -B_{x,y} : x,y\in X\rangle = \langle q\rangle \cong \bbZ$, so $\cH_\beta(f)$ is a sequence of $\bbZ$-graded $\bbZ$-modules.
    Thus, as can be seen by following Khovanov's construction \cite{khovanov2000categorification} or Bar-Natan's paper \cite{bar2002khovanov}, $\cH_\beta(f)$ is the Khovanov homology invariant of $L$, denoted $\Kh(L)$.
    
    Now, if instead of $\bbZ[t,t^{-1}]$ we had started with some other ring $R$, the biquandle bracket value is the Jones polynomial \emph{evaluated} at $q$.
    If $q$ is a torsion element of $R^\times$, then we will not recover the full $\Kh(L)$.
    Indeed, let $k\in \bbN$ be minimal such that $q^k=1$. Then $\cH_\beta(f)$ is $\Kh(L)$ considered as a $\bbZ/k \bbZ$-graded $\bbZ$-module, as described in remark \ref{rmk_quotientgrading}.
    Since $\Kh(L)$ is an invariant of links $L$, the object we obtain will still be invariant of links.
    
    Also, since $\beta$ is a constant biquandle bracket, $\cH_\beta(f)$ will not depend on the coloring $f$ of $L$, only the link itself.
    So $\Bh_\beta(L)$ is a multiset containing a quotient of $\Kh(L)$ with multiplicity equal to the number of $X$-colorings of $L$.
    This shows that $\Bh_\beta(L)$ is an invariant of links when $\beta$ is a constant biquandle bracket.
\end{example}

\begin{remark}
\label{rmk_nottruecategorification}
$\Bh_\beta(L)$ is \emph{not} a true categorification of the biquandle bracket invariant $\Phi_X^\beta(L)$ in the sense that $\Kh(L)$ is a categorification of the Jones polynomial.
This is because the Euler characteristic of $\Bh_\beta(L)$ is actually $\gdim(S) \cdot \Phi_X^\beta(L)$ (this is a slight abuse of notation, since the Euler characteristic of $\Bh_\beta(L)$ is a multiset of formal sums, this $\Phi_X^\beta(L)$ is actually a multiset of finite formal sums that give the usual $\Phi_X^\beta(L)$ when evaluated as in remark \ref{rmk_evaluation}).
Note that $\gdim(S)$ is a formal sum of the elements of $G$.

If $G$ is the trivial group, then $\gdim(S)=1$ and $\Phi_X^\beta(L)$ can be recovered from $\Bh_\beta(L)$. However, if $|G|>1$, then for any $g\in G$ we have $g\cdot \gdim(S)=\gdim(S)$, and so the factor of $\gdim(S)$ cannot be cancelled from the graded Euler characteristic. 

Note that $|G|=1$ iff $q=q_{x,y}$ for all $x,y\in X$, and in \cite{hoffer2019structure} this was shown to be the case iff $\beta$ is the product of a constant bracket and a 2-cocycle.
If $|G|>1$, one can still ask how much information about the biquandle bracket is retained in the Euler characteristic of $\Bh_\beta(L)$.
These questions are strongly related to questions about the canonical biquandle 2-cocycle associated with a biquandle bracket, which will be discussed in the next section.
\end{remark}

\subsection{Proof of Invariance and the Canonical 2-Cocycle}

To prove that $\Bh_\beta(L)$ is an invariant of oriented links $L$ it is sufficient to prove that $\cH_\beta(f)$ is an invariant of $X$-colored oriented links $f$.
For this, we will actually prove that $\cH_\beta(f)$ is isomorphic to a shift of a quotient of $\Kh(L)$, the original Khovanov homology invariant of $L$.
The shift itself turns out to be an invariant of $X$-colored links, and it is obtained from a particular biquandle $2$-cocycle.
In this way, it becomes possible to canonically assign a biquandle $2$-cocycle to an biquandle bracket.
First we need a lemma.

\begin{lemma}
\label{lem_isomorphism}
Let $H,G$ be abelian groups with $\bbZ[G]$ an $H$-graded ring. Suppose also that each $g\in G$ has well defined degree. Let $M$ be an $H$-graded $\bbZ[G]$-module.
Then $M \cong g M$ as $H$-graded $\bbZ[G]$-modules for all $g \in G$.
\end{lemma}

\begin{proof}
Suppose $M = \bigoplus_{h \in H} M_h$. Since each $g\in G$ has well defined degree, the map $\deg:G\to H$ is a group homomorphism. Thus $\deg(g^{-1})=\deg(g)^{-1}$.
Then the ``multiplication-by-$g$'' map is an isomorphism $M_h \to M_{\deg(g)h}$, since the ``multiplication-by-$g^{-1}$'' map is the inverse.
Then
\[
    g M = \bigoplus_{h \in H} g M_h = \bigoplus_{h \in H} M_{\deg(g)h} \cong \bigoplus_{h' \in H} M_{h'} = M,
\]
since $\deg(g) H = H$ because $H$ is a group.
\end{proof}

Consider the group homomorphism $\bbZ\to R^\times$ defined by $1\mapsto q$ (writing $\bbZ$ additively here). Define $\Kh_\beta(L)$ to be $\Kh(L)$ considered as a sequence of $R^\times$-graded $\bbZ$-modules, as in remark \ref{rmk_quotientgrading}. 
Both $\cH_\beta(f)$ and $\Kh_\beta(L)$ are sequences of homologies of cochain complexes.
The claim is now that $\cH_\beta(f) \cong \Kh_\beta(L)\{Z_\beta(f)\}$, where $Z_\beta(f)$ is an invariant of $X$-colored links $f$.
Moreover, the isomorphism between $\cH_\beta(f)$ and $\Kh_\beta(L)\{Z_\beta(f)\}$ is actually an isomorphism of the cochain complexes they are derived from.
Since both cochain complexes are direct sums of cube diagrams, it suffices to construct an isomorphism of cubes, by which we mean a collection of isomorphisms from each vertex in one cube to the corresponding vertex of the other cube such that the squares thus adjoined to each edge commute.
We henceforth construct the isomorphism of the cube used to construct $\cH_\beta(f)$ with the cube used to construct $\Kh_\beta(L)$, before the final writhe-correcting shift for each cube.

Starting with the cube in the construction of $\cH_\beta(f)$, shift all of the $R^\times$-gradings by the grading shift applied to the left-most vertex i.e. $\prod_{\tau_+} A_{x_\tau,y_\tau}\prod_{\tau_-} B_{x_\tau,y_\tau}^{-1}$.
This yields a cube where the left-most vertex is a tensor power of $M$ with no shift applied, and each other vertex will be a tensor power of $M$ with a shift that can be written as a product of the elements $q_{x,y} \in R^\times$.
The maps in the cube are unchanged by this shift, since the coefficient on a map is the quotient of the shifts on the target and domain spaces.

For each vertex $M^{\otimes k} \{ \prod_{i=1}^j q_{x_i,y_i} \}$ of the shifted cube, the ``multiplication-by-$(\prod_{i=1}^j q^{-1} q_{x_i,y_i})$'' map is an isomorphism $M^{\otimes k} \{ \prod q_{x,y} \} \to M^{\otimes k} \{ q^j \}$ by lemma \ref{lem_isomorphism}.
And these maps assemble into an isomorphism of cubes because each edge in the cube is a map of the form $q_{z,w}^{-1} q \, \Delta$ or $q_{z,w}^{-1} q \, m$, and the following two squares commute:
\begin{center}
\begin{tikzcd}
{M^{\otimes k}\{ \prod_{i=1}^j q_{x_i,y_i} \}} \arrow[rr, "{q_{z,w}^{-1} q \, \Delta}"] \arrow[dd, "{\text{multiplication-by-} \left(\prod_{i=1}^j q^{-1} q_{x_i,y_i} \right)}"'] &  & {M^{\otimes (k+1)} \{q_{z,w} \prod_{i=1}^j q_{x_i,y_i} \}} \arrow[dd, "{\text{multiplication-by-} \left( q^{-1} q_{z,w} \prod_{i=1}^j q^{-1} q_{x_i,y_i} \right)}"] \\
& & \\
M^{\otimes k}\{q^j\} \arrow[rr, "\Delta"]                                                                                                                     &  & M^{\otimes (k+1)} \{ q^{j+1} \}                                                                                                                      
& & \\
{M^{\otimes k}\{ \prod_{i=1}^j q_{x_i,y_i} \}} \arrow[rr, "{q_{z,w}^{-1} q \, m}"] \arrow[dd, "{\text{multiplication-by-} \left(\prod_{i=1}^j q^{-1} q_{x_i,y_i} \right)}"'] & & {M^{\otimes (k-1)} \{q_{z,w} \prod_{i=1}^j q_{x_i,y_i} \}} \arrow[dd, "{\text{multiplication-by-} \left( q^{-1} q_{z,w} \prod_{i=1}^j q^{-1} q_{x_i,y_i} \right)}"] \\
& & \\
M^{\otimes k}\{q^j\} \arrow[rr, "m"]                                                                                                                    & & M^{\otimes (k-1)} \{ q^{j+1} \}                                                                                                                     
\end{tikzcd}
\end{center}
Thus we have an isomorphism from the shifted cube to a cube with only powers of $q$ as shifts and with no coefficients on the maps between vertices.
This is almost exactly the cube in the construction of $\Kh(L)$ (see \cite{khovanov2000categorification} or \cite{bar2002khovanov}).
However, $q$ may be torsion in $R^\times$, and in this cube $M$ is a free $S$-module, as opposed to a free $\bbZ$-module like in the construction of $\Kh(L)$. Let $M_\bbZ$ denote the Frobenius algebra used in the construction of $\Kh_\beta(L)$. Then $M$ and $M_\bbZ$ have the same basis, and so $M=\bigoplus_{g\in G} g\cdot M_\bbZ$. It follows then that the cohomology of this cochain complex is $\bigoplus_{g\in G} g\cdot \Kh_\beta(L)\cong \Kh_\beta(L)\{\gdim(S)\}$, where the isomorphism is given by identifying $g\cdot \Kh_\beta(L)$ with $\Kh_\beta(L)\{g\}$.

Thus, what we truly obtain is an isomorphism between $\cH_\beta(f)$ and $\Kh_\beta(L)$ shifted by
\[
\prod_{\tau_+} A_{x_\tau,y_\tau}\prod_{\tau_-} B_{x_\tau,y_\tau}^{-1},
\]
by $\gdim(S)$, and also by the difference in the respective writhe correction factors. For $\cH_\beta(f)$, this writhe correction factor is again $(-1)^{n_-}w^{n_--n_+}= (-1)^{n_-}(A_{x_0,x_0}q^{-1})^{n_--n_+}$. For $\Kh_\beta(L)$, this factor is $q^{n_+-2n_-}$ (from \cite{khovanov2000categorification}). Thus the difference in these factors is
\[
    (-1)^{n_-}(A_{x_0,x_0}q^{-1})^{n_--n_+}q^{2n_--n_+}= (-1)^{n_-}A_{x_0,x_0}^{n_--n_+}q^{n_-} = A_{x_0,x_0}^{-n_+}B_{x_0,x_0}^{n_-}
\]
Putting all of this together, we arrive at

\begin{theorem}
\label{thm_factorization}
$\cH_\beta(f) \cong \Kh_\beta(L) \{ Z_\beta(f) \}$, where the shift is given by
\[
    Z_\beta(f) = \left( \prod_{\tau^+} A_{x_\tau,y_\tau} A_{x_0,x_0}^{-1} \right) \left( \prod_{\tau^-} B_{x_\tau,y_\tau}^{-1} B_{x_0,x_0} \right) \cdot \gdim(S).
\]
\end{theorem}

Since all of the factors in the two products defining $Z_\beta(f)$ are invertible elements of $R$, and since $\gdim(S)$ is a formal sum of all elements in $G \leq R^\times$, we can view this shift as taking values in the abelian quotient group $R^\times / G$.
To see that $Z_\beta(f)$ is in fact the value of a biquandle 2-cocycle invariant on $f$; consider the map $\phi_\beta : X \times X \to R^\times / G$, defined by $\phi_\beta(x,y) = A_{x,y} A_{x_0, x_0}^{-1} \cdot G$. For any $x\in X$ we have 
\[
    A_{x,x}^2B_{x,x}^{-1} = -w = A_{x_0,x_0}^2B_{x_0,x_0}^{-1}
\]
Thus 
\[
    \phi_\beta(x,x) = A_{x,x}A_{x_0,x_0}^{-1}\cdot G = B_{x,x}A_{x,x}^{-1}A_{x_0,x_0}B_{x_0,x_0}^{-1}\cdot G = q_{x,x}q^{-1}\cdot G = G
\]
And for any $x,y,z\in X$
\begin{align*}
    \phi_\beta(x,y) \cdot \phi_\beta(y,z) \cdot \phi_\beta\left(x\untri y, z\ovtri y\right) &= A_{x_0,x_0}^{-3} A_{x,y} A_{y,z} A_{x \untri y, z \ovtri y} \cdot G \\ &= A_{x_0,x_0}^{-3} A_{x,z} A_{y \ovtri x, z \ovtri x} A_{x \untri z, y \untri z} \cdot G \\ &= \phi_\beta(x,z) \cdot \phi_\beta\left(y\ovtri x,z\ovtri x\right) \cdot \phi_\beta\left(x\untri z, y\untri z\right)
\end{align*}
using biquandle bracket axiom (iii). Thus $\phi_\beta$ is a biquandle 2-cocycle. We call $\phi_\beta$ the \emph{canonical 2-cocycle} associated with $\beta$.

It is clear that on positive crossings $\tau^+$, the value of $Z_\beta(f)$ is derived from $\phi_\beta$.
For negative crossings, we observe that
\[
    A_{x,y} A_{x_0, x_0}^{-1} B_{x,y}^{-1} B_{x_0, x_0} = q_{x,y}^{-1} q \in G
\]
so that $B_{x,y}^{-1} B_{x_0, x_0}$ is the inverse of $A_{x,y} A_{x_0, x_0}^{-1}$, mod $G$. Hence $Z_\beta(f) = \prod_{\tau} \phi_\beta\left(x_\tau,y_\tau\right)^{\epsilon(\tau)}$, and so

\begin{prop}
\label{prop_Zinvariance}
    $Z_\beta(f)$ is a biquandle 2-cocycle invariant of $X$-colored links $f$.
\end{prop}

Since $\Kh_\beta(L)$ is also an invariant of oriented links $L$, this proposition (combined with Theorem 1) shows that $\cH_\beta(f)$ is an invariant of $X$-colored links $f$.
Therefore, we find

\begin{corollary}
\label{cor_invariant}
The multiset $\Bh_\beta(L)$ of colored biquandle bracket homology values is an invariant of oriented uncolored links $L$.
\end{corollary}

\begin{remark}
\label{rmk_canonical}
$\phi_\beta$ is canonical in the sense that it only depends on $\beta$. Namely, $\phi_\beta$ doesn't depend at all on the choice of the distinguished element $x_0\in X$. Though $x_0$ was used in the definition of $G = \langle q_{x,y}^{-1}q : x,y\in X\rangle$, for any $x,y,v,z\in X$ we have
\[
    (q_{x,y}^{-1}q)\left(q_{v,z}^{-1}q\right)^{-1} = q_{x,y}^{-1}q_{v,z}\in G.
\]
Hence $G=\langle q_{x,y}^{-1}q_{v,z} : x,y,v,z\in X\rangle$, and so $G$ does not depend on $x_0$. If $y_0\in X$, then 
\[
    G = \phi_\beta(y_0,y_0) = A_{y_0,y_0} A_{x_0,x_0}^{-1} \cdot G
\]
so $A_{y_0,y_0} \cdot G = A_{x_0,x_0} \cdot G$. Hence $\phi_\beta$ doesn't depend on the choice of $x_0$.
\end{remark}
\begin{remark}
\label{rmk_incomparable}
As with any biquandle 2-cocycle invariant, we can view $\phi_\beta$ as an invariant of uncolored links $L$ by letting $Z_\beta(L)$ denote the multiset of values $Z_\beta(f)$ as $f$ varies across all $X$-colorings of $L$.
A reasonable question to ask is how the power of $Z_\beta(L)$ compares with the power of the original biquandle bracket invariant $\Phi_X^\beta(L)$.
In fact, the strength of these two invariants can differ quite drastically depending on $\beta$, and in general they are incomparable.

For example, if $\beta$ is the Jones polynomial as in example \ref{ex_jones}, then $\phi_\beta$ is trivial, and so $Z_\beta(L)$ is the trivial invariant.
In this case $\Phi_X^\beta(L)$ is stronger than $Z_\beta(L)$.
On the other hand, let $\phi:X^2\to H$ be any 2-cocycle, and let $\beta$ take values in $\left(\bbZ/2\bbZ\right)[H]$ with $A_{x,y}=\phi(x,y)=B_{x,y}$ for all $x,y\in X$.
Then $\phi_\beta = \phi$, but $\delta = 0$. Thus $\beta(f) = 0$ for any $X$-colored link $f$, so $\Phi_X^\beta(L)=\Phi_X^\bbZ(L)$, the biquandle counting invariant.
Meanwhile, $Z_\beta(L)$ is the invariant corresponding to $\phi$.
This is always an enhancement of $\Phi_X^\bbZ(L)$, and depending on $\phi$, is often a strict enhancement. 
\end{remark}

\subsection{An Example Computation}
We will compute explicitly the biquandle bracket homology of a Hopf link with a specific biquandle bracket, and we will also compute the canonical 2-cocycle associated to this biquandle bracket.
These calculations can be done by hand or with the aid of a computer programs, such as the Mathematica packages which can be found at \href{http://www.vilas.us/biquandles/}{vilas.us/biquandles}.
We will work with the following biquandle $X$, which has $2$ elements, (which we will call $a$ and $b$ instead of $1$ and $2$ to avoid confusion with other notation), with the following operation tables (this is the biquandle whose operations ``swap'' the first argument):
\[
    \untri : \begin{bmatrix} b & b \\ a & a \end{bmatrix} \qquad \qquad \ovtri : \begin{bmatrix} b & b \\ a & a \end{bmatrix}
\]

We will consider the $X$-bracket $\beta = (A,B)$ which take values in $R=\mathbb{Z}[t] / (3, 1 + t+ t^2)$, with the following presentation matrix (we use boldface to denote that these elements are in the quotient ring, and for example the element $\mathbf{2}$ of $R$ is not the same as the integer $2$).
\[
    \left[
    \begin{array}{cc|cc}
        \mathbf{1} & \mathbf{2} \mathbf{t} & \mathbf{2} \mathbf{t} & \mathbf{t}  \\
        \mathbf{1} & \mathbf{1} & \mathbf{2} & \mathbf{2} \mathbf{t} \\
    \end{array}
    \right]
\]
Let $x_0 = a \in X$ be the distinguished biquandle element, so that
\[
q = - A_{a,a}^{-1} B_{a,a} = - (\mathbf{1})^{-1} \mathbf{2} \mathbf{t} = - \mathbf{2} \mathbf{t} = \mathbf{t}
\]
One can check that the values of $q_{x,y}$ for every other pair $x,y \in X$ is equal to either $\mathbf{1}$ or $\mathbf{t}$, so the generators of the group $G \leq R^\times$ are $\mathbf{1}^{-1} \mathbf{t} = \mathbf{t}$ and $\mathbf{t}^{-1} \mathbf{t} = \mathbf{1}$. Since $\mathbf{t}^3 = \mathbf{1}$ (by the factorization $t^3-1 = (t-1)(1+t+t^2)$), we have $G = \left\{ \mathbf{1}, \mathbf{t}, \mathbf{t}^2 \right\}$. Thus the ring $S = \bbZ[G] = \bbZ \mathbf{1} \oplus \bbZ \mathbf{t} \oplus \bbZ \mathbf{t}^2$ is a $R^\times$-graded ring, with $\deg(\mathbf{1}) = \mathbf{1}$, $\deg(\mathbf{t}) = \mathbf{t}$, and $\deg(\mathbf{t}^2) = \mathbf{t}^2$.
Therefore, we have
\[
    \gdim(S) = \mathbf{1} + \mathbf{t} + \mathbf{t}^2
\]
which is a formal sum of the elements of the group $R^\times$, consisting of one summand for each element of the group $G$.
Note that $\gdim(S)$ \emph{does not} equal $\mathbf{0}$, since it is not an element of $R$.

Now suppose $L$ is the Hopf link, oriented so that both crossings are negative.
To calculate $\Bh_\beta(L)$, we will invoke Theorem \ref{thm_factorization} for each $X$-coloring of the $L$.
However, for all colorings we will need to determine $\Kh_\beta(L)$, which does not depend on the coloring.
It can be easily computed with computer software that the standard Khovanov Homology of $L$ is the following sequence of $\bbZ$-graded $\bbZ$-modules:
\[
    \begin{tabular}{c | c c c}
      & -2 & -1 & 0 \\
    \hline
    -6 & $\bbZ$ & 0 & 0  \\
    -4 & $\bbZ$ & 0 & 0  \\
    -2 & 0 & 0 & $\bbZ$   \\
    0 & 0 & 0 & $\bbZ$
    \end{tabular}
\]
Each column is one $\bbZ$-graded $\bbZ$-module, and nonzero graded homology groups appear only at indices $-2$ and $0$.
The integers on the left-hand side are the $\bbZ$-gradings of the modules in the sequence.
Now, using the map $\varphi : \bbZ \to R^\times$ given by $1 \mapsto q = \mathbf{t}$, we can consider this as a sequence of $R^\times$-graded $\bbZ$-modules instead.
Now since $\mathbf{t}^{-1} = \mathbf{t}^2$, we have $\mathbf{t}^{-6} = \mathbf{t}^{12} = \mathbf{1} = \mathbf{t}^0$, we have $\mathbf{t}^{-4} = \mathbf{t}^8 = \mathbf{t}^2$, and we have $\mathbf{t}^{-2} = \mathbf{t}^4 = \mathbf{t}$. So we obtain the following sequence of $R^\times$-graded $\bbZ$-modules, and this is the value of $\Kh_\beta(L)$:
\[
    \begin{tabular}{c | c c c}
      & -2 & -1 & 0 \\
    \hline
    $\mathbf{1}$ & $\bbZ$ & 0 & $\bbZ$ \\
    $\mathbf{t}$ & 0 & 0 & $\bbZ$ \\
    $\mathbf{t}^2$ & $\bbZ$ & 0 & 0
    \end{tabular}
\]

Since for each $X$-coloring $f$ of $L$, we have $\mathcal{H}_\beta(f) = \Kh_\beta(L)\{Z_\beta(f)\}$, it just remains to compute $Z_\beta(f)$ for each $X$-coloring $f$ of $L$.
For this, let's first compute $\phi_\beta$, the canonical biquandle $2$-cocycle associated with the biquandle bracket $\beta$.
Then for each coloring $f$ we can obtain $Z_\beta(f)$ as the biquandle $2$-cocycle invariant associated with $\phi_\beta$, computed on the coloring $f$.

Recall that $\phi_\beta(x,y) = A_{x,y} A^{-1}_{x_0,x_0} \cdot G$ for each $x,y \in X$ (note $\phi_\beta$ takes values in the quotient $R^\times / G$).
Since $x_0 = a$ and $A_{a,a} = \mathbf{1}$, this means $\phi_\beta(x,y) = A_{x,y} \cdot G$ for each $x, y \in X$.
Since $\mathbf{2} \mathbf{t} \cdot G = \mathbf{2}\cdot G\neq G$, we have the following presentation matrix for $\phi_\beta$:
\[
    \begin{bmatrix}
        G & \mathbf{2} \cdot G \\
        G & G
    \end{bmatrix}
\]

Now we could examine the four $X$-colorings of $L$ and use the above presentation for $\phi_\beta$ to compute $Z_\beta(f)$ for each of them.
However, this biquandle $2$-cocycle is of the form described in Example \ref{ex_linking}, and the Hopf link is a link with two components.
So, since the linking number of the two components is $1$, the multiset of values of $Z_\beta(f)$ as $f$ varies across all $X$-colorings of $L$ is simply $\{ G, G, \mathbf{2} \cdot G, \mathbf{2} \cdot G \}$.
Now, to turn these into valid shifts we will reinterpret the symbol $G$ as $\gdim(S) = \mathbf{1} + \mathbf{t} + \mathbf{t}^2$, and thus the multiset of shifts becomes
\[
    \left\{ \gdim(S), \gdim(S), \mathbf{2} \gdim(S), \mathbf{2} \gdim(S) \right\}.
\]
Now notice that
\begin{align*}
    \Kh_\beta(L) \{ \gdim(S) \} &= \Kh_\beta(L) \{ \mathbf{1} + \mathbf{t} + \mathbf{t}^2 \} = \Kh_\beta(L) \{ \mathbf{1} \} \oplus \Kh_\beta(L) \{ \mathbf{t} \} \oplus \Kh_\beta(L) \{ \mathbf{t}^2 \} \\
    &= \left( \begin{tabular}{c | c c c}
      & -2 & -1 & 0 \\
    \hline
    $\mathbf{1}$ & $\bbZ$ & 0 & $\bbZ$ \\
    $\mathbf{t}$ & 0 & 0 & $\bbZ$ \\
    $\mathbf{t}^2$ & $\bbZ$ & 0 & 0
    \end{tabular} \right) \oplus \left( \begin{tabular}{c | c c c}
      & -2 & -1 & 0 \\
    \hline
    $\mathbf{t}$ & $\bbZ$ & 0 & $\bbZ$ \\
    $\mathbf{t}^2$ & 0 & 0 & $\bbZ$ \\
    $\mathbf{1}$ & $\bbZ$ & 0 & 0
    \end{tabular} \right) \oplus \left( \begin{tabular}{c | c c c}
      & -2 & -1 & 0 \\
    \hline
    $\mathbf{t}^2$ & $\bbZ$ & 0 & $\bbZ$ \\
    $\mathbf{1}$ & 0 & 0 & $\bbZ$ \\
    $\mathbf{t}$ & $\bbZ$ & 0 & 0
    \end{tabular} \right) \\
    &= \left( \begin{tabular}{c | c c c}
      & -2 & -1 & 0 \\
    \hline
    $\mathbf{1}$ & $\bbZ^2$ & 0 & $\bbZ^2$ \\
    $\mathbf{t}$ & $\bbZ^2$ & 0 & $\bbZ^2$ \\
    $\mathbf{t}^2$ & $\bbZ^2$ & 0 & $\bbZ^2$
    \end{tabular} \right) ,
\end{align*}
and similarly
\[
    \Kh_\beta(L) \{ \mathbf{2} \gdim(S) \} = \left( \begin{tabular}{c | c c c}
      & -2 & -1 & 0 \\
    \hline
    $\mathbf{2}$ & $\bbZ^2$ & 0 & $\bbZ^2$ \\
    $\mathbf{2t}$ & $\bbZ^2$ & 0 & $\bbZ^2$ \\
    $\mathbf{2t}^2$ & $\bbZ^2$ & 0 & $\bbZ^2$
    \end{tabular} \right).
\]
Thus, the value of $\Bh_\beta(L)$, which is the multiset of $\mathcal{H}_\beta(f) = \Kh_\beta(L) \{ Z_\beta(f) \}$, is

\begin{align*}
\left\{
    \left( \begin{tabular}{c | c c c}
      & -2 & -1 & 0 \\
    \hline
    $\mathbf{1}$ & $\bbZ^2$ & 0 & $\bbZ^2$ \\
    $\mathbf{t}$ & $\bbZ^2$ & 0 & $\bbZ^2$ \\
    $\mathbf{t}^2$ & $\bbZ^2$ & 0 & $\bbZ^2$
    \end{tabular} \right) ,
    \left( \begin{tabular}{c | c c c}
      & -2 & -1 & 0 \\
    \hline
    $\mathbf{1}$ & $\bbZ^2$ & 0 & $\bbZ^2$ \\
    $\mathbf{t}$ & $\bbZ^2$ & 0 & $\bbZ^2$ \\
    $\mathbf{t}^2$ & $\bbZ^2$ & 0 & $\bbZ^2$
    \end{tabular} \right),
    \left( \begin{tabular}{c | c c c}
      & -2 & -1 & 0 \\
    \hline
    $\mathbf{2}$ & $\bbZ^2$ & 0 & $\bbZ^2$ \\
    $\mathbf{2t}$ & $\bbZ^2$ & 0 & $\bbZ^2$ \\
    $\mathbf{2t}^2$ & $\bbZ^2$ & 0 & $\bbZ^2$
    \end{tabular} \right),
    \left( \begin{tabular}{c | c c c}
      & -2 & -1 & 0 \\
    \hline
    $\mathbf{2}$ & $\bbZ^2$ & 0 & $\bbZ^2$ \\
    $\mathbf{2t}$ & $\bbZ^2$ & 0 & $\bbZ^2$ \\
    $\mathbf{2t}^2$ & $\bbZ^2$ & 0 & $\bbZ^2$
    \end{tabular} \right)
    \right\}.
\end{align*}

Now, what if we take the graded Euler characteristic (element-wise) of this invariant value?
First of all,
\begin{align*}
    \chi \left( \begin{tabular}{c | c c c}
      & -2 & -1 & 0 \\
    \hline
    $\mathbf{1}$ & $\bbZ^2$ & 0 & $\bbZ^2$ \\
    $\mathbf{t}$ & $\bbZ^2$ & 0 & $\bbZ^2$ \\
    $\mathbf{t}^2$ & $\bbZ^2$ & 0 & $\bbZ^2$
    \end{tabular} \right) = (-1)^{-2} \left( 2 \mathbf{1} + 2 \mathbf{t} + 2 \mathbf{t}^2 \right) + (-1)^0 \left( 2 \mathbf{1} + 2 \mathbf{t} + 2 \mathbf{t}^2 \right) = 4(\mathbf{1} + \mathbf{t} + \mathbf{t}^2) = 4\gdim(S).
\end{align*}
Similarly, we have
\begin{align*}
    \chi \left( \begin{tabular}{c | c c c}
      & -2 & -1 & 0 \\
    \hline
    $\mathbf{2}$ & $\bbZ^2$ & 0 & $\bbZ^2$ \\
    $\mathbf{2t}$ & $\bbZ^2$ & 0 & $\bbZ^2$ \\
    $\mathbf{2t}^2$ & $\bbZ^2$ & 0 & $\bbZ^2$
    \end{tabular} \right) = 4(\mathbf{2} + \mathbf{2t} + \mathbf{2} \mathbf{t}^2) = 4\cdot \mathbf 2 \gdim(S).
\end{align*}
Thus the multiset of graded Euler characteristics is
\[
    \left\{ 4(\mathbf{1} + \mathbf{t} + \mathbf{t}^2), 4(\mathbf{1} + \mathbf{t} + \mathbf{t}^2), 4(\mathbf{2} + \mathbf{2t} + \mathbf{2} \mathbf{t}^2), 4(\mathbf{2} + \mathbf{2t} + \mathbf{2} \mathbf{t}^2) \right\} = \gdim(S) \cdot \left\{ 4\cdot \mathbf 1, 4\cdot \mathbf 1, 4\cdot \mathbf 2, 4\cdot \mathbf 2 \right\}.
\]
After evaluating the formal sums $4\cdot \mathbf 1$ and $4\cdot \mathbf 2$, as mentioned in remark \ref{rmk_evaluation}, we can visually recover that the value of the biquandle bracket invariant $\Phi_X^\beta(L)$ is $\{ \mathbf{1}, \mathbf{1}, \mathbf{2}, \mathbf{2} \}$.
Indeed, this is the correct value of $\Phi_X^\beta(L)$, as can be calculated separately.
However, another perfectly valid way of writing the above multiset is
\[
    \gdim(S) \cdot \{ 4 \cdot \mathbf{t}, 4\cdot \mathbf{t}^2 , 4\cdot \mathbf{2} \mathbf{t}, 4\cdot \mathbf{2} \mathbf{t}^2\},
\]
since $\mathbf{t}, \mathbf{t}^2 \in G$, and $\mathbf{2} \mathbf{t}, \mathbf{2} \mathbf{t}^2 \in \mathbf{2} \cdot G$.
If presented with this description, one might think that the value of the biquandle bracket invariant $\Phi_X^\beta(L)$ is actually $\{ \mathbf{t}, \mathbf{t}^2, \mathbf{2} \mathbf{t}, \mathbf{2} \mathbf{t}^2 \}$, which is incorrect.
This demonstrates, as mentioned in remark \ref{rmk_nottruecategorification}, that we \emph{cannot} consistently recover $\Phi_X^\beta(L)$ from $\Bh_\beta(L)$ in all cases by simply taking the Euler characteristic.
Thus $\Bh_\beta(L)$ is not a true categorification of $\Phi_X^\beta(L)$, at least not in the same way that Khovanov homology is a categorification of the Jones polynomial.

\section{Conclusions and Further Questions}
\label{sec_questions}
The biquandle bracket generalizes the Jones polynomial $J(L)$, which Khovanov homology categorifies.
Our goal was to generalize Khovanov homology to a categorification of biquandle brackets and obtain the invariant $\bigstar$ in the following diagram:
\begin{center}
    \begin{tikzcd}
    \bigstar_\beta(L) \arrow[ddd, "{\text{Take } \beta = \left[ q \mid q^{-1} \right]}"'] \arrow[rrrr, "\text{Take Euler Characteristic}"] &  &  &  &  \Phi_X^\beta(L) \vphantom{\mathcal{H}_\beta(L)} \arrow[ddd, "{\text{Take } \beta = \left[ q \mid q^{-1} \right]}"]  \\
                                                                                     &  &  &  &                                                                                                                           \\
                                                                                     &  &  &  &                                                                                                                           \\
    \Kh(L) \arrow[rrrr, "\text{Take Euler Characteristic}"']  &  &  &  & J(L)                                                                   
    \end{tikzcd}
\end{center}
However, the top arrow in this diagram fails to hold in general with our invariant $\Bh_\beta(L)$:
\begin{center}
    \begin{tikzcd}
    \Bh_\beta(L) \arrow[ddd, "{\text{Take } \beta = \left[ q \mid q^{-1} \right]}"'] \arrow[rrrr, dotted, "\text{Take Euler Characteristic}"] &  &  &  &  \Phi_X^\beta(L) \vphantom{\mathcal{H}_\beta(L)} \arrow[ddd, "{\text{Take } \beta = \left[ q \mid q^{-1} \right]}"]  \\
                                                                                     &  &  &  &                                                                                                                           \\
                                                                                     &  &  &  &                                                                                                                           \\
    \Kh(L) \arrow[rrrr, "\text{Take Euler Characteristic}"']  &  &  &  & J(L)                                                                   
    \end{tikzcd}
\end{center}
This opens up a few questions.

First, how does the invariant $\Bh_\beta(L)$ compare in power to the biquandle bracket invariant $\Phi_X^\beta(L)$?
Specifically, we can ask how the two separate pieces, $\Kh_\beta(L)$ (the quotient of Khovanov homology) and $Z_\beta(L)$ (the invariant obtained from the canonical 2-cocycle associated to $\beta$) compare in power to $\Phi_X^\beta(L)$.

As far as $\Kh_\beta(L)$ goes, it is simply weaker than the previously-known $\Kh(L)$.
And the invariant $Z_\beta(L)$ is actually easier to compute than $\Bh_\beta(L)$.
So the most important comparison to make is between $Z_\beta(L)$ and $\Phi_X^\beta(L)$.
However, as seen in remark \ref{rmk_incomparable}, these two invariants are incomparable in general.
Thus more work remains to be done in comparing $\Bh_\beta(L)$ with $\Phi_X^\beta(L)$.

Second, does the invariant $\bigstar$ exist?
The construction we have outlined seems (to the authors) to be the most natural step away from Khovanov homology toward a categorification of biquandle brackets, but technical limitations prevent the construction from truly categorifying all biquandle brackets.
Is it possible, with more advanced techniques, to create an invariant that simultaneously categorifies every biquandle bracket and generalizes Khovanov homology?

\nocite{elhamdadi2015quandles}
\bibliographystyle{plain}
\bibliography{main}

\end{document}